\documentclass[11pt]{amsart}

\usepackage{latexsym}
\usepackage{amsmath}
\usepackage{amssymb}

\usepackage{epsfig}

\usepackage{enumerate}

\newtheorem{theorem}{Theorem}[section]
\newtheorem{lemma}[theorem]{Lemma}

\newcommand{\cD}{{\mathcal D}}

\newcommand{\cN}{{\mathcal N}}
\newcommand{\cP}{{\mathcal P}}
\newcommand{\cQ}{{\mathcal Q}}
\newcommand{\cT}{{\mathcal T}}

\newcommand{\bd}{{\mathbf d}}
\usepackage{color}

\title[Recovering Tree-child Networks]{Recovering Tree-Child Networks from Shortest Inter-taxa Distance Information}

\author[Bordewich, Huber, Moulton, and Semple]{Magnus Bordewich \and Katharina T. Huber \and Vincent Moulton \and Charles Semple}

\address{Department of Computer Science, Durham University, Durham, DH1 3LE, United Kingdom}
\email{m.j.r.bordewich@durham.ac.uk}

\address{School of Computing Sciences, University of East Anglia, Norwich NR4 7TJ, United Kingdom}
\email{k.huber@uea.ac.uk}

\address{School of Computing Sciences, University of East Anglia, Norwich NR4 7TJ, United Kingdom}
\email{v.moulton@uea.ac.uk}

\address{School of Mathematics and Statistics, University of Canterbury, Christchurch, New Zealand}
\email{charles.semple@canterbury.ac.nz}

\thanks{The second and third authors were supported by the London Mathematical Society. The fourth author was supported by the New Zealand Marsden Fund.}



\date{\today}

\begin{document}

\maketitle

\begin{abstract}
Phylogenetic networks are a type of leaf-labelled, acyclic, directed graph used by biologists to represent the evolutionary history of species whose past includes reticulation events. A phylogenetic network is tree-child if each non-leaf vertex is the parent of a tree vertex or a leaf. Up to a certain equivalence, it has been recently shown that, under two different types of weightings, edge-weighted tree-child networks are determined by their collection of distances between each pair of taxa. However, the size of these collections can be exponential in the size of the taxa set. In this paper, we show that, if we ignore redundant edges, the same results are obtained with only a quadratic number of inter-taxa distances by using the shortest distance between each pair of taxa. The proofs are constructive and give cubic-time algorithms in the size of the taxa sets for building such weighted networks.
\end{abstract}

\section{Introduction}

Distance-based methods collectively provide fundamental tools for the reconstruction and analysis of phylogenetic (evolutionary) trees. Two of the most popular and longstanding distance-based methods are UPGMA~\cite{sok58} and Neighbor Joining~\cite{sai87}. Loosely speaking, these methods take as input a distance matrix $\cD$ on a set $X$ of taxa, whose entries are the distances between pairs of taxa in $X$, and return an edge-weighted phylogenetic tree on $X$ that best represents $\cD$. Distances between taxa could, for example, measure the time since the two taxa separated from a common ancestor. Typically, the property underlying any distance-based method is the following: if $\cT$ is a phylogenetic tree whose edges are assigned a positive real-valued weight, then the pairwise distances between taxa is sufficient to determine $\cT$ and its edge weighting~\cite{bun74, sim69, zar65}. This property is frequently referred to as the Tree-Metric Theorem (see~\cite[Theorem~7.2.6]{sem03}).

While there exist numerous distance-based methods for reconstructing and analysing phylogenetic trees that aim to explicitly represent the treelike evolution of species, there are only a small number of such methods for phylogenetic networks. Yet, phylogenetic networks are necessary to accurately represent the ancestral history of sets of taxa whose evolution includes non-treelike (reticulate) evolutionary processes such as hybridisation and lateral gene transfer. As a step towards developing practical distance-based methods for reconstructing and analysing phylogenetic networks, in this paper we are interested in establishing analogues of the Tree-Metric Theorem for edge-weighted phylogenetic networks. In particular, we establish two such analogues for the increasingly prominent class of tree-child networks. Briefly, we show that, under two types of weightings, edge-weighted tree-child networks on $X$ with no redundant edges can be reconstructed from the pairwise shortest distances between taxa in time polynomial in the size of $X$. The first type of weighting induces an ultrametric, while the second type of weighting has the property that the pair of edges directed into a reticulation (\textit{i.e.} a vertex of in-degree two) have equal weight. We envisage that these results could lead to practical algorithms to construct phylogenetic networks from distance data.

In related work, Chan \textit{et al.}~\cite{cha06} take a matrix of inter-taxa distances and reconstruct an ultrametric galled network having the property that there is a path between each pair of taxa having the same length as that given in the matrix, if such a network exists. Willson~\cite{wil12} studied the problem of determining a phylogenetic network given the average distance between each pair of taxa, where each reticulation assigns a probability to the two edges directed into it. It is shown in~\cite{wil12} that such distances are enough to determine a phylogenetic network with a single reticulation cycle in polynomial time. Bordewich and Semple~\cite{bor16a}, the original starting point for this paper, showed that (unweighted) tree-child networks can be reconstructed from the multi-set of distances between taxa in polynomial time. Other methods have been developed for building phylogenetic networks from distance data (see, for example, \cite{wil07b, yu15}) but these use different approaches to associate a distance to a network than the ones presented here. In addition, Huber and Scholz~\cite{hub17} considered the problem of reconstructing phylogenetic networks from so-called symbolic distances, and it would be interesting to see if our new results could extend to such distances. Two other papers (specificially \cite{bor17, bor16b}) that are particularly relevant to the work of our paper, are discussed in more detail in the next section.

Although distance based methods may be considered not as accurate as other methods such as Maximum Likelihood, they still have an important role in studying large datasets and gaining quick results when exploring data. This role may even be more significant for phylogenetic networks, where the complexity of inferring optimal solutions is even higher than for phylogenetic trees. The next section details some background and gives the statements of the two main results.

\section{Main Results}

Throughout the paper, $X$ denotes a non-empty finite set. A {\em phylogenetic network $\cN$ on $X$} is a rooted acyclic directed graph with no parallel edges satisfying the following:
\begin{enumerate}[(i)]
\item the unique root has out-degree two;

\item the set $X$ is the set of vertices of out-degree zero, each of which has in-degree one; and

\item all other vertices either have in-degree one and out-degree two, or in-degree two and out-degree one.
\end{enumerate}
If $|X|=1$, then we additionally allow the directed graph consisting of the single vertex in $X$ to be a phylogenetic network. The vertices in $X$ are called {\em leaves}, while the vertices of in-degree one and out-degree two are {\em tree vertices}, and the vertices of in-degree two and out-degree one are {\em reticulations}. An edge directed into a reticulation is a {\em reticulation edge}; all other edges are {\em tree edges}. An element in $X$ is an {\em outgroup} if its parent is the root of $\cN$. Furthermore, an edge $e=(u, v)$ is a {\em shortcut} if there is a directed path in $\cN$ from $u$ to $v$ avoiding $e$. Note that, necessarily, a shortcut is a reticulation edge. In the literature, shortcuts are also known as {\em redundant} edges. Ignoring the weighting of the edges, Fig.~\ref{weighted}(i) shows a phylogenetic network with root $\rho$, outgroup $r$, and $X=\{r, x_1, x_2, x_3, x_4, x_5, x_6\}$. It has exactly two reticulations, but no shortcuts.

\begin{figure}
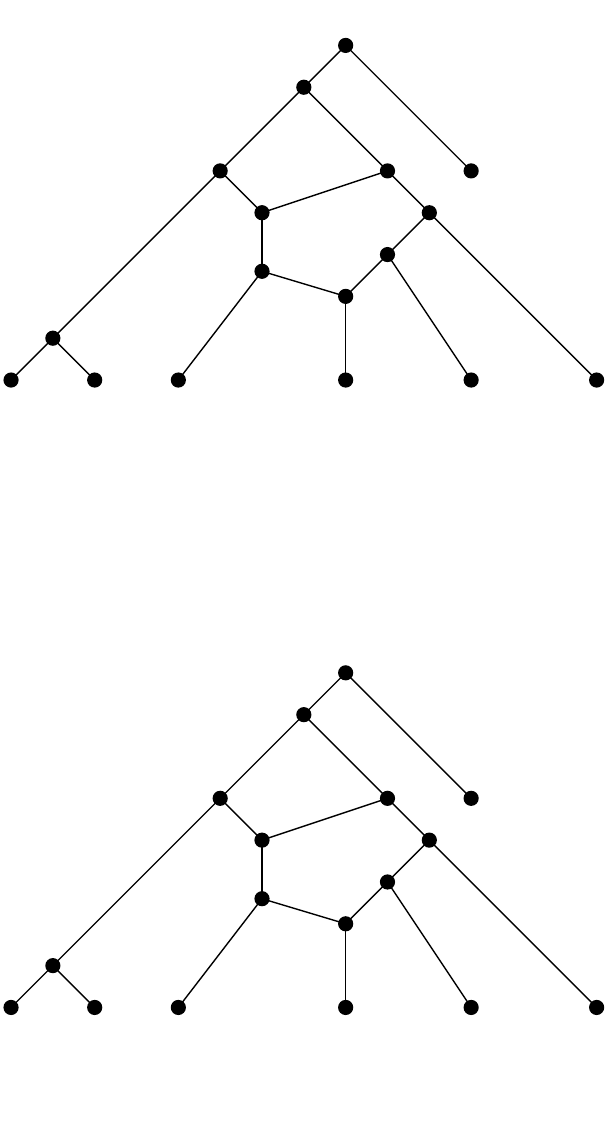
\caption{Two (weighted) phylogenetic networks on $\{r, x_1, x_2, x_3, x_4, x_5, x_6\}$.}
\label{weighted}
\end{figure}

Let $\cN$ be a phylogenetic network. We say $\cN$ is {\em tree-child} if each non-leaf vertex in $\cN$ is the parent of either a tree vertex or a leaf. Moreover, $\cN$ is {\em normal} if it is tree-child and has no shortcuts. For example, the phylogenetic network in Fig.~\ref{weighted}(i) is normal. As with all figures in this paper, edges are directed down the page. Tree-child networks and normal networks were introduced by Cardona et al.~\cite{car09} and Willson~\cite{wil07a, wil10}, respectively.

Let $\cN$ be a phylogenetic network on $X$, and let $E$ denote the edge set of $\cN$. A {\em weighting} of $\cN$ is a function $w:E\rightarrow \mathbb R$ that assigns edges a non-negative real-valued weight such that tree edges are assigned a strictly positive weight. Thus reticulation edges are allowed weight zero. Let $(\cN, w)$ be a weighted phylogenetic network, and let $v$ and $v'$ be vertices in $\cN$. An {\em up-down path} $P$ from $v$ to $v'$ is an underlying path
$$v=u_0, u_1, u_2, \ldots, u_k=v',$$
where, for some $i$,
$$(u_i, u_{i-1}), (u_{i-1}, u_{i-2}), \ldots, (u_1, v)$$
and
$$(u_i, u_{i+1}), (u_{i+1}, u_{i+2}), \ldots, (u_{k-1}, v')$$
are edges in $\cN$. The {\em length} of $P$ is the sum of the weights of the edges in $P$ and, provide $v$ and $v'$ are distinct, $u_i$ is the {\em peak} of $P$. In Fig.~\ref{weighted}(i), there are exactly two up-down paths joining $x_3$ and $x_5$. One path has weight~$21$, while the other path has weight~$14$. As in this example, in this paper we are interested in the {\em inter-taxa distances} in $(\cN, w)$, that is, the lengths of the up-down paths connecting leaves in $X$. We next state two recent results~\cite{bor16b, bor17}. Both results prompted the two main results in this paper.

Let $(\cN, w)$ be a weighted phylogenetic network on $X$ and, for all $x, y\in X$, let $\cP_{x, y}$ be the set of up-down paths from $x$ to $y$ in $(\cN, w)$. The {\em multi-set of distances from $x$ to $y$}, denoted $\cD_{x, y}$, is the multi-set of the lengths of the paths in $\cP_{x, y}$. Similarly, the {\em set of distances from $x$ to $y$}, denoted $\overline{\cD}_{x, y}$, is the set of the lengths of the paths in $\cP_{x, y}$. Note that $\cD_{x, y}=\cD_{y, x}$, $\cD_{x, x}=\{0\}$, $\overline{\cD}_{x, y}=\overline{\cD}_{y, x}$, and $\overline{\cD}_{x, x}=\{0\}$ for all $x, y\in X$. 
The {\em multi-set distance matrix $\cD$ of $(\cN, w)$} is the $|X|\times |X|$ matrix whose $(x, y)$-th entry is $\cD_{x, y}$ (respectively, the {\em set distance matrix $\overline{\cD}$ of $(\cN, w)$} is the $|X|\times |X|$ matrix whose $(x, y)$-th entry is $\overline{\cD}_{x, y}$) for all $x, y\in X$, in which case we say  $\cD$ (resp.\ $\overline{\cD}$) is {\em realised by $(\cN, w)$}.

Let $(\cN, w)$ be a weighted phylogenetic network on $X$ and let $\cD$ be the multi-set distance matrix of $(\cN, w)$. The underlying problem we are investigating is determining how much $\cD$ says about $(\cN, w)$. The following highlights that, even with tree edges having positive weights, the weighting of $(\cN, w)$ cannot be determined exactly. Let $u$ be a reticulation in $\cN$ with parents $p_u$ and $q_u$, and let $v$ be the unique child of $u$. We can change the weighting of the edges incident with $u$ without changing $\cD$. In particular, provided the sum of the weights of $(p_u, u)$ and $(u, v)$ equal
$$w(p_u, u)+w(u, v)$$
and the sum of the weights of $(q_u, u)$ and $(u, v)$ equal
$$w(q_u, u)+w(u, v),$$
we can change the weights of $(p_u, u)$, $(q_u, u)$, and $(u, v)$ (keeping the weights of all other edges the same) to construct a different weighting, $w'$ say, so that $\cD$ is realised by $(\cN, w')$. We refer to this change as {\em re-weighting the edges at a reticulation} of $\cN$. For example, in Fig.~\ref{weighted}, $(\cN, w')$ has been obtained from $(\cN, w)$ by reweighting the edges at both reticulations. If $\cN$ has an outgroup $r$, a similar occurrence happens with the two edges incident with the root $\rho$ of $\cN$. In particular, now let $u$ be the child of $\rho$ that is not $r$. Note that $u$ is a tree vertex. Then, provided the weights of $(\rho, r)$ and $(\rho, u)$ sum to
$$w(\rho, r)+w(\rho, u),$$
we can change the weights of $(\rho, r)$ and $(\rho, u)$ (keeping the weights of all other edges the same) to construct a different weighting, $w''$ say, so that $\cD$ is realised by $(\cN, w'')$. We refer to this change as {\em re-weighting the edges at the root} of $\cN$.

Let $(\cN, w)$ and $(\cN_1, w_1)$ be two weighted phylogenetic networks on $X$. We say $(\cN, w)$ and $(\cN_1, w_1)$ are {\em equivalent} if $\cN$ is isomorphic to $\cN_1$, and $w_1$ can be obtained from $w$ by re-weighting the edges at each reticulation and re-weighting the edges at the root. For example, the weighted phylogenetic networks in Fig.~\ref{weighted} are equivalent.

\begin{figure}
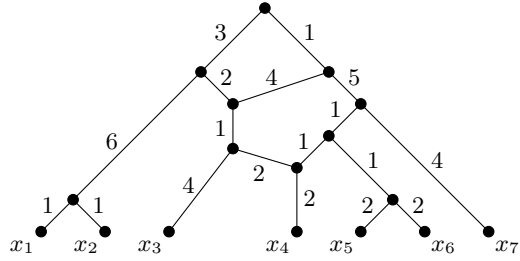
\caption{A phylogenetic network with an equidistant weighting.}
\label{equidistant}
\end{figure}

We now state the two recent results. A weighting $w$ of a phylogenetic network $\cN$ is {\em equidistant} if the length of all paths starting at the root and ending at a leaf are the same length. While a weighting $w$ of a phylogenetic network is a {\em reticulation-pair weighting} if, for each reticulation $v$ in $(\cN, w)$, the two edges directed into $v$ have the same weight. The weightings of the phylogenetic networks shown in Fig.~\ref{weighted} are reticulation-pair weightings. The weighting of the phylogenetic network shown in Fig.~\ref{equidistant} is an equidistant weighting. The following theorems are established in \cite{bor16b} and \cite{bor17}, respectively.

\begin{theorem}
Let $\overline{\cD}$ be the set distance matrix of an equidistant-weighted tree-child network $(\cN, w)$ on $X$. Then, up to equivalence, $(\cN, w)$ is the unique such network realising $\overline{\cD}$, in which case a member of its equivalence class can be found from $\overline{\cD}$ in $O(|X|^4)$ time.
\label{tokac1}
\end{theorem}

\begin{theorem}
Let $\cD$ be the multi-set distance matrix with distinguished element $r$ of a reticulation-pair weighted tree-child network $(\cN, w)$ on $X$ with outgroup $r$. Then, up to equivalence, $(\cN, w)$ is the unique such network realising $\cD$, in which case a member of its equivalence class can be found from $\cD$ in time $O(|\cD|^2)$.
\label{tokac2}
\end{theorem}

Theorems~\ref{tokac1} and~\ref{tokac2} are somewhat surprising given that, in the size of the leaf set, it is possible to have exponentially-many up-down paths connecting leaves in a tree-child network. For example, Fig.~\ref{many} shows a normal network with $2n+1$ leaves $x_1, x_2, \ldots, x_{2n+1}$ in which there are $2^n$ distinct up-down paths connecting $x_1$ and $x_{2n+1}$. Nevertheless, the fact that all such paths are considered is not satisfactory. The next two theorems are the two main results of this paper. They show that, up to shortcuts, we can retain the outcomes of Theorems~\ref{tokac1} and~\ref{tokac2} by knowing only a quadratic number of inter-taxa distances. 

\begin{figure}
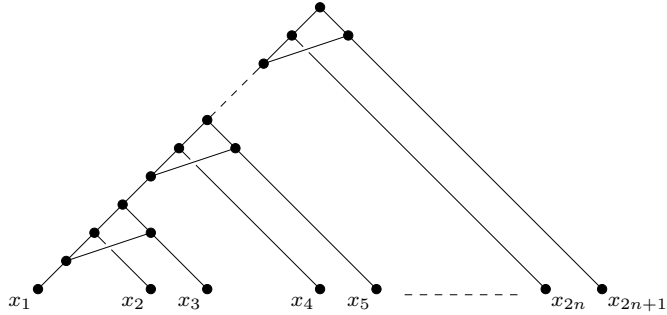
\caption{A normal network on $\{x_1, x_2, \ldots, x_{2n+1}\}$. There are $2^n$ distinct up-down paths connecting $x_1$ and $x_{2n+1}$.}
\label{many}
\end{figure}

Let $(\cN, w)$ be a weighted phylogenetic network on $X$. The {\em minimum distance matrix  $\cD_{\min}$ of $(\cN, w)$} is the $|X|\times |X|$ matrix whose $(x, y)$-th entry is the minimum length of an up-down path joining $x$ and $y$ for all $x, y\in X$. We denote the $(x, y)$-th entry in $\cD_{\min}$ by $d_{\min}(x, y)$.

\begin{theorem}
Let $\cD_{\min}$ be the minimum distance matrix of an equidistant-weighted normal network $(\cN, w)$ on $X$. Then, up to equivalence, $(\cN, w)$ is the unique such network realising $\cD_{\min}$, in which case a member of its equivalence class can be found from $\cD_{\min}$ in $O(|X|^3)$ time.
\label{child1}
\end{theorem}

Now let $(\cN, w)$ be a weighted phylogenetic network on $X\cup \{r\}$, where $r$ is an outgroup. For the purposes of the next theorem, the minimum distance matrix $\cD_{\min}$ of $(\cN, w)$ is the $|X|\times |X|$ matrix whose $(x, y)$-th entry is $d_{\min}(x, y)$ for all $x, y\in X$. Furthermore, the {\em maximum distance outgroup vector}, denoted $\bd_{\max}$, is the vector of length $|X|$ whose $x$-th coordinate is the maximum length of an up-down path joining $r$ and $x$ for all $x\in X$. We denote the $x$-th coordinate in $\bd_{\max}$ by $d_{\max}(r, x)$.

\begin{theorem}
Let $\cD_{\min}$ and $\bd_{\max}$ be the minimum distance matrix and maximum distance outgroup vector of a reticulation-pair weighted normal network $(\cN, w)$ on $X\cup \{r\}$, where $r$ is an outgroup. Then, up to equivalence, $(\cN, w)$ is the unique such network realising $\cD_{\min}$ and $\bd_{\max}$, in which case a member of its equivalence class can be found from $\cD_{\min}$ and $\bd_{\max}$ in $O(|X|^3)$ time.
\label{child2}
\end{theorem}

\noindent It is easily seen that it is not possible to extend Theorems~\ref{child1} and~\ref{child2} to tree-child networks as the distance information given in the hypothesis of these theorems is insufficient to determine shortcuts. However, these results do hold for temporal tree-child networks as such networks have no shortcuts and are therefore normal. For the definition of temporal phylogenetic network, see~\cite{bar06}.

The rest of the paper is organised as follows. The next section contains some necessary preliminaries. In particular, it contains the constructions of the three operations that underlie the inductive proofs of Theorems~\ref{child1} and~\ref{child2}. These proofs, including explicit descriptions of the associated algorithms, are given in Sections~\ref{proof1} and~\ref{proof2}, respectively.

\section{Preliminaries}

Let $\cN$ be a phylogenetic network on $X$ and let $\{s, t\}$ be a $2$-element subset of $X$. Denote the unique parents of $s$ and $t$ by $p_s$ and $p_t$, respectively. We call $\{s, t\}$ a {\em cherry} if $p_s=p_t$, that is, the parents of $s$ and $t$ are the same. Furthermore, we call $\{s, t\}$ a {\em reticulated cherry} if either $p_s$ or $p_t$, say $p_t$, is a reticulation and $p_s$ is a parent of $p_t$, in which case $t$ is the {\em reticulation leaf} of the reticulated cherry. Observe that $p_s$ is necessarily a tree vertex. To illustrate, in Fig.~\ref{weighted}(i), $\{x_1, x_2\}$ is a cherry, while $\{x_3, x_4\}$ is a reticulated cherry in which $x_4$ is the reticulation leaf. In the same figure, $\{x_4, x_5\}$ is also a reticulated cherry. The next lemma is well known for tree-child networks (for example, see~\cite{bor16a}). The restriction to normal networks is immediate. We will used it freely throughout the paper.

\begin{lemma}
Let $\cN$ be a normal network on $X$. Then
\begin{enumerate}[{\rm (i)}]
\item If $|X|=1$, then $\cN$ consists of the single vertex in $X$, while if $|X|=2$, say $X=\{s, t\}$, then $\cN$ consists of the cherry $\{s, t\}$.

\item If $|X|\ge 2$, then $\cN$ has either a cherry or a reticulated cherry.
\end{enumerate}
\end{lemma}

\noindent In addition to using the last lemma freely, we will also use freely the following observation. If $\cN$ is a normal network and $u$ is a vertex of $\cN$, then there is a directed path from $u$ to a leaf avoiding reticulations except perhaps $u$.

Let $(\cN, w)$ be a weighted phylogenetic network on $X$. We now describe three operations on $(\cN, w)$. The first and second operations underlie the inductive approach we take to prove Theorem~\ref{child1}, while the first and third operations underlie the inductive approach we take to prove Theorem~\ref{child2}. Let $\{s, t\}$ be  $2$-element subset of $X$, and denote the parents of $s$ and $t$ by $p_s$ and $p_t$, respectively. First suppose that $\{s, t\}$ is a cherry. Let $g_s$ denote the parent of $p_s$. Then {\em reducing $t$} is the operation of deleting $t$ and its incident edge, suppressing $p_s$, and setting the weight of the resulting edge $(g_s, s)$ to be
$$w(g_s, p_s)+w(p_s, s).$$
Now suppose that $\{s, t\}$ is a reticulated cherry, in which $t$ is the reticulation leaf. Let $g_s$ denote the parent of $p_s$, and let $g_t$ denote the parent of $p_t$ that is not $p_s$. Then {\em cutting $\{s, t\}$} is the operation of deleting $(p_s, p_t)$, suppressing $p_s$ and $p_t$, and setting the weight of the resulting edge $(g_s, s)$ to be
$$w(g_s, p_s)+w(p_s, s)$$
and the edge $(g_t, t)$ to be
$$w(g_t, p_t)+w(p_t, t).$$
Lastly, if $g_t$ is a tree vertex, then {\em isolating $\{s, t\}$} is the operation of deleting $(g_t, p_t)$, suppressing $g_t$ and $p_t$, and setting the weight of the resulting edge $(p_s, t)$ to be
$$w(p_s, p_t)+w(p_t, t)$$
and the edge $(g'_t, h)$ to be
$$w(g'_t, g_t)+w(g_t, h),$$
where $g'_t$ is the parent of $g_t$ and $h$ is the child of $g_t$ that is not $p_t$. To illustrate the last two operations, consider Fig.~\ref{cutting}. The weighted phylogenetic network $(\cN_1, w_1)$ has been obtained from the weighted phylogenetic network in Fig.~\ref{equidistant} by cutting $\{x_3, x_4\}$, while $(\cN_2, w_2)$ has been obtained from the same network by isolating $\{x_3, x_4\}$.

\begin{figure}
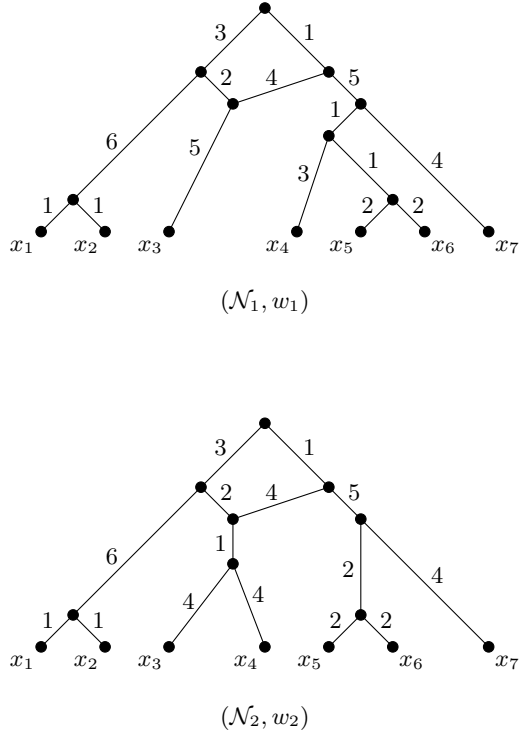
\caption{The weighted phylogenetic networks $(\cN_1, w_1)$ and $(\cN_2, w_2)$ obtained from the weighted phylogenetic network in Fig.~\ref{equidistant} by cutting $\{x_3, x_4\}$ and isolating $\{x_3, x_4\}$, respectively.}
\label{cutting}
\end{figure}

The proof of the next lemma is straightforward and omitted.

\begin{lemma}
Let $(\cN, w)$ be a weighted normal network. Let $\{s, t\}$ be a cherry or a reticulated cherry of $\cN$. If $(\cN', w')$ is obtained from $(\cN, w)$ by reducing $t$ if $\{s, t\}$ is a cherry, or by cutting or isolating $\{s, t\}$ if $\{s, t\}$ is a reticulated cherry, then $\cN$ is a normal network. Furthermore, if $w$ is an equidistant (resp.\ reticulation-pair) weighting, then $w'$ is an equidistant (resp.\ reticulation-pair) weighting.
\label{unchanged}
\end{lemma}

We next describe three operations on distance matrices that parallel the operations of reducing, cutting, and isolating. Let $\cD$ be a distance matrix on $X$ with each entry consisting of a single value, that is, $\cD$ is an $|X|\times |X|$ matrix whose $(x, y)$-th entry is denoted $d(x, y)$. Let $\{s, t\}$ be a $2$-element subset of $X$. Although the choice of $\{s, t\}$ appears arbitrary, in the sections that follow, $\{s, t\}$ will always be a cherry in the case of the first operation or a reticulated cherry in which $t$ is the reticulation leaf in the case of the second and third operations. Furthermore, in the sections that follow, the operations will always be well defined. First, let $\cD'$ be the distance matrix on $X'=X-\{t\}$ obtained from $\cD$ by setting
$$d'(x, y)=d'(y, x)=d(x, y)$$
for all $x, y\in X'$. We say that $\cD'$ has been obtained from $\cD$ by {\em reducing $t$ in $\cD$}.

For the second operation, let
$$X_t=\{x\in X-\{s, t\}: d(t, x)\neq d(s, x)\}$$
and let $\delta=\min\{d(t, x): x\in X_t\}$. 
Furthermore, let
$$X_{\delta}=\{x\in X_t: d(t, x)=\delta\}.$$
Intuitively, $X_t$ are those leaves whose shortest path to $t$ does not go through the parent of $s$, and $X_\delta$ are those members of $X_t$ at minimum distance from $t$. Now let $\cD'$ be the distance matrix on $X$ obtained from $\cD$ by setting
$$d'(x, y)=d'(y, x)=d(x, y)$$
for all $x, y\in X-\{t\}$, setting $d'(t, t)=0$, and setting
$$d'(t, y)=d'(y, t)=\max\{d(x, y): x\in X_{\delta}-\{y\}\}$$
if $\max\{d(x, y): x\in X_{\delta}-\{y\}\}\ge \delta$, and
$$d'(t, y)=d'(y, t)=\delta$$
otherwise for all $y\in X-\{t\}$. We say that $\cD'$ has been obtained from $\cD$ by {\em cutting $\{s, t\}$ in $\cD$}. Intuitively, elements of $X_\delta$ are being used as a proxy for $t$ in determining the minimum distances from $t$ to members of $X_t-X_\delta$ in $\cD'$; the distance from $t$ to any member of $X_\delta$ remains $\delta$.

For the last operation, let $r$ be a distinguished element in $X$, and let
$$\gamma=d(r, t)-d(r, s).$$
Intuitively, $\gamma$ is the difference in length of the edge from the parent of $s$ to $s$ and the path from the parent of $s$ to $t$. Let $\cD'$ be the distance matrix on $X$ obtained from $\cD$ by setting
$$d'(x, y)=d'(y, x)=d(x, y)$$
for all $x, y\in X-\{t\}$,
$$d'(t, x)=d'(x, t)=d(s, x)+\gamma$$
for all $x\in X-\{s, t\}$, and
$$d'(t, s)=d'(s, t)=d(t, s).$$
We say that $\cD'$ has been obtained from $\cD$ by {\em isolating $\{s, t\}$ in $\cD$}.

\section{Proof of Theorem~\ref{child1}}
\label{proof1}

In this section, we establish Theorem~\ref{child1}. We begin with two lemmas.

\begin{lemma}
Let $(\cN, w)$ be a equidistant-weighted normal network on $X$, where $|X|\ge 2$. Let $\cD_{\min}$ be the minimum distance matrix of $(\cN, w)$, and let $\{s, t\}$ be a $2$-element subset of $X$ such that
$$d_{\min}(s, t)=\min\{d_{\min}(x, y): x, y\in X\}.$$
Then $\{s, t\}$ is either a cherry or a reticulated cherry of $(\cN, w)$. Moreover,
\begin{enumerate}[{\rm (i)}]
\item The set $\{s, t\}$ is a cherry of $(\cN, w)$ if $d_{\min}(s, x)=d_{\min}(t, x)$ for all $x\in X-\{s, t\}$.

\item Otherwise, $\{s, t\}$ is a reticulated cherry of $(\cN, w)$ in which $t$ is the reticulation leaf precisely if $d_{\min}(s, x) > d_{\min}(t, x)$ for some $x\in X-\{s, t\}$.
\end{enumerate}
\label{cherry1}
\end{lemma}

\begin{proof}
Let $p_s$ and $p_t$ denote the parents of $s$ and $t$, respectively. If $p_s$ and $p_t$ are both reticulations, then, as $\cN$ is normal and $w$ equidistant, it is easily seen that there is an element $x\in X-\{s, t\}$ such that either $d_{\min}(s, t) > d_{\min}(s, x)$ or $d_{\min}(s, t) > d_{\min}(t, x)$; a contradiction ($x$ is a descendant of a tree vertex on the shortest up-down path from $s$ to $t$ that is not the peak of that up-down path). Thus, either $p_s$ or $p_t$ is a tree vertex. Without loss of generality, we may assume that $p_s$ is a tree vertex. Let $u$ denote the child of $p_s$ that is not $s$. If $u$ is a tree vertex, then, as $w$ is equidistant,
$$d_{\min}(s, t) > d_{\min}(a, b) \geq \min\{d_{\min}(x, y): x, y\in X\},$$
where $\{a,b\}$ is a cherry or reticulated cherry and $a, b$ are descendants of $u$; a contradiction. Therefore either $u$ is a leaf or a reticulation. If $u$ is a leaf, then, as $w$ is equidistant, $u=t$, in which case, $\{s, t\}$ is a cherry. If $u$ is a reticulation, then, as $\cN$ is normal and
$$d_{\min}(s, t) = \min\{d_{\min}(x, y): x, y\in X\},$$
it follows that the unique child of $u$ is $t$, and so $\{s, t\}$ is a reticulated cherry. Since $\cN$ has no shortcuts and $w$ is equidistant, it is easily checked that $\{s, t\}$ is a reticulated cherry in which $t$ is the reticulation leaf if and only if $d(s, t) > d(t, x)$ for some $x\in X-\{s, t\}$. The lemma immediately follows.
\end{proof}

\begin{lemma}
Let $(\cN, w)$ be an equidistant-weighted normal network on $X$, where $|X|\ge 2$. Let $\cD_{\min}$ be the minimum distance matrix of $(\cN, w)$, and let $\{s, t\}$ be a $2$-element subset of $X$ such that
$$d_{\min}(s, t)=\min\{d_{\min}(x, y): x, y\in X\}.$$
Then the distance matrix obtained from $\cD_{\min}$ by reducing $t$, if $\{s, t\}$ is a cherry, and by cutting $\{s, t\}$, if $\{s, t\}$ is a reticulated cherry in which $t$ is the reticulation leaf, is the minimum distance matrix $\cD'_{\min}$ realised by the weighted network $(\cN', w')$ obtained from $(\cN, w)$ by reducing $t$, if $\{s, t\}$ is a cherry, and cutting $\{s, t\}$, if $\{s, t\}$ is a reticulated cherry in which $t$ is the reticulation leaf.
\label{cutting1}
\end{lemma}

\begin{proof}
By Lemma~\ref{cherry1}, $\{s, t\}$ is either a cherry or a reticulated cherry. Furthermore, by the same lemma, if $\{s, t\}$ is a reticulated cherry, we may assume, without loss of generality, that $t$ is the reticulation leaf. Also, by Lemma~\ref{unchanged}, $(\cN', w')$ is an equidistant-weighted normal network regardless of which of the two stated ways it is obtained from $(\cN, w)$. If $\{s, t\}$ is a cherry, then it is clear that the distance matrix obtained from $\cD_{\min}$ by reducing $t$ is the minimum distance matrix of $(\cN', w')$. Therefore, suppose that $\{s, t\}$ is a reticulated cherry, in which case $(\cN', w')$ is obtained from $(\cN, w)$ by cutting $\{s, t\}$. Let $\cD'$ be the distance matrix obtained from $\cD_{\min}$ by cutting $\{s, t\}$. We will show that $\cD'$ is the minimum distance matrix $\cD'_{\min}$ of $(\cN', w')$.

Let $p_s$ and $p_t$ denote the parents of $s$ and $t$, respectively, in $(\cN, w)$. Since the only up-down paths in $(\cN, w)$ joining elements in $X$ that traverse the edge $(p_s, p_t)$ involve $t$, it follows that $d'(x, y)=d'_{\min}(x, y)$ for all $x, y\in X-\{t\}$. Thus, to complete the proof, it suffices to show that $d'(t, x)=d'_{\min}(t, x)$ for all $x\in X-\{t\}$.

Let $g_t$ denote the parent of $p_t$ that is not $p_s$. Since $\cN$ has no shortcuts, $g_t$ is not an ancestor of $s$. Therefore, as $\cN$ is normal, there is a (directed) path from $g_t$ to a leaf, $\ell$ say, containing no reticulations, where $\ell\not\in \{s, t\}$. Note that, in what follows, we never determine $\ell$ but its existence underlies the rest of the proof.

Let
$$X_t=\{x\in X-\{s, t\}: d_{\min}(t, x)\neq d_{\min}(s, x)\}.$$
Thus, if $x\in X_t$, then every minimum length up-down path in $(\cN, w)$ joining $t$ and $x$ must traverse the edge $(g_t, p_t)$. Observe that $\ell\in X_t$ as the edge directed into $g_t$, which is a tree edge, has positive weight and every up-down path from $s$ to $\ell$ traverses this edge and therefore $g_t$, so $d_{\min}(t, \ell) < d_{\min}(s, \ell)$. Now let
$$\delta=\min\{d_{\min}(t, x): x\in X_t\}$$
and let $X_{\delta}$ denote the subset of $X_t$ consisting of those elements $x$ such that $d_{\min}(t, x)=\delta$, that is,
$$X_{\delta}=\{x\in X_t: d_{\min}(t, x)=\delta\}.$$
Again, observe that $\ell\in X_{\delta}$, and so the elements in $X_{\delta}$ are descendants of $g_t$.

Let $y\in X-\{t\}$. We next determine whether or not $y$ is a descendant of $g_t$. First note that $|X_{\delta}|=1$ if and only if $g_t$ is the parent of a leaf, in which case $\ell$ is the only leaf apart from $t$ that is a descendant of $g_t$. So assume $|X_{\delta}|\ge 2$. We establish two claims:
\begin{enumerate}[(i)]
\item If
$$\max\{d_{\min}(x, y): x\in X_{\delta}-\{y\}\}\ge \delta,$$
then $y$ is not a descendant of $g_t$.

\item If
$$\max\{d_{\min}(x, y): x\in X_{\delta}-\{y\}\} < \delta,$$
then $y$ is a descendant of $g_t$.
\end{enumerate}
To see (i) and (ii), if $y$ is a descendant of $g_t$, then $d_{\min}(x, y) < \delta$ for all $x\in X_{\delta}-\{y\}$, so
$$\max\{d_{\min}(x, y): x\in X_{\delta}-\{y\}\} < \delta.$$
On the other hand, if $y$ is not a descendant of $g_t$, then
$$\max\{d_{\min}(x, y): x\in X_{\delta}-\{y\}\}\ge d_{\min}(\ell, y)\ge d_{\min}(t, \ell) = \delta.$$

It follows from (i) and (ii) that, for all $y\in X-\{t\}$, if $y$ is not a descendant of $g_t$, then
$$d'_{\min}(t, y)=\max\{d_{\min}(x, y): x\in X_{\delta}-\{y\}\},$$
otherwise
$$d'_{\min}(t, y)=\delta.$$
Hence $d'(t, x)=d'_{\min}(t, x)$ for all $x\in X-\{t\}$, thereby completing the proof of the lemma.
\end{proof}

We next prove the uniqueness part of Theorem~\ref{child1}.

\begin{proof}[Proof of the uniqueness part of Theorem~\ref{child1}]
The proof is by induction on the sum of the number $n$ of leaves and the number $k$ of reticulations in $(\cN, w)$. If $n+k=1$, then $n=1$ and $k=0$, and $(\cN, w)$ consists of the single vertex in $X$, and so uniqueness holds. If $n+k=2$, then, as $\cN$ is normal, $n=2$ and $k=0$, and $(\cN, w)$ consists of two leaves attached to the root. Again, uniqueness holds as $w$ is equidistant and so the weights of the edges incident with the leaves is fixed. Now suppose that $n+k\ge 3$, so $n\ge 2$, and that the uniqueness holds for all equidistant-weighted normal networks for which the sum of the number of leaves and the number of reticulations is at most $n+k-1$.

Let $\{s, t\}$ be a $2$-element subset of $X$ such that
$$d_{\min}(s, t)=\min\{d_{\min}(x, y): x, y\in X\}.$$
By Lemma~\ref{cherry1}, $\{s, t\}$ is either a cherry or a reticulated cherry of $(\cN, w)$. If $\{s, t\}$ is a reticulated cherry, then, by the same lemma, we can determine from $\cD_{\min}$ which of $s$ and $t$ is the reticulation leaf. Thus, without loss of generality, we may assume that $t$ is the reticulation leaf. Depending on whether $\{s, t\}$ is a cherry or a reticulated cherry, let $(\cN', w')$ and $\cD'$ be the weighted network and distance matrix obtained from $(\cN, w)$ and $\cD_{\min}$ by reducing $t$ or cutting $\{s, t\}$, respectively. By Lemma~\ref{cutting1}, $\cD'$ is the minimum distance matrix of $(\cN', w')$. Since $(\cN', w')$ either has $n-1$ leaves and $k$ reticulations if $\{s, t\}$ is a cherry, or $n$ leaves and $k-1$ reticulations if $\{s, t\}$ is a reticulated cherry, it follows by the induction assumption that, up to equivalence, $(\cN', w')$ is the unique equidistant-weighted normal network with minimum distance matrix $\cD'$.

Let $(\cN_1, w_1)$ be an equidistant-weighted normal network on $X$ with minimum distance matrix $\cD_{\min}$. By Lemma~\ref{cherry1}, $\{s, t\}$ is either a cherry or a reticulated cherry in $(\cN_1, w_1)$. Indeed, by the same lemma, $\{s, t\}$ is a cherry in $(\cN_1, w_1)$ precisely if it is a cherry in $(\cN, w)$. First assume that $\{s, t\}$ is a cherry in $(\cN, w)$. Then $\{s, t\}$ is a cherry in $(\cN_1, w_1)$. Let $(\cN'_1, w'_1)$ be the equidistant-weighted normal network obtained from $(\cN_1, w_1)$ by reducing $t$. By Lemma~\ref{cutting1}, $\cD'$ is the minimum distance matrix of $(\cN'_1, w'_1)$ and so, by the induction assumption, $(\cN'_1, w'_1)$ and $(\cN', w')$ are equivalent. Using this equivalence and considering $d_{\min}(s, t)$, it is easily seen that $(\cN_1, w_1)$ and $(\cN, w)$ are equivalent.

Now assume that $\{s, t\}$ is a reticulated cherry in $(\cN, w)$. Then $\{s, t\}$ is a reticulated cherry in $(\cN_1, w_1)$ where, by Lemma~\ref{cherry1}, $t$ is a reticulation leaf. Let $(\cN'_1, w'_1)$ be the equidistant-weighted normal network obtained from $(\cN_1, w_1)$ by cutting $\{s, t\}$. Since $\cD_{\min}$ is the minimum distance matrix of $(\cN_1, w_1)$, it follows by Lemma~\ref{cutting1} that $\cD'$ is the minimum distance matrix of $(\cN'_1, w'_1)$. Therefore, by the induction assumption, $(\cN'_1, w'_1)$ and $(\cN', w')$ are equivalent. By again considering $d_{\min}(s, t)$, it is now easily deduced that $(\cN_1, w_1)$ and $(\cN, w)$ are equivalent. This completes the proof of the uniqueness part of the theorem.
\end{proof}

\subsection{The Algorithm}

Let $(\cN, w)$ be an equidistant-weighted normal network on $X$ and let $\cD_{\min}$ denote the minimum distance matrix of $(\cN, w)$. We next give an algorithm which takes as input $X$ and $\cD_{\min}$, and returns a weighted network $(\cN_0, w_0)$ equivalent to $(\cN, w)$. Its correctness is essentially established in proving the uniqueness part of Theorem~\ref{child1} and so a formal proof of this is omitted. However, its running time is given at the end of this section. Called {\sc Equidistant Normal}, the algorithm works as follows.
\begin{enumerate}[1.]
\item If $|X|=1$, then return the weighted phylogenetic network consisting of the single vertex in $X$.

\item If $|X|=2$, say $X=\{s, t\}$, then return the weighted phylogenetic network with exactly two leaves $s$ and $t$ adjoined to the root by edges each with weight $\frac{1}{2}d_{\min}(s, t)$.

\item Else, find a $2$-element subset $\{s, t\}$ of $X$ such that
$$d_{\min}(s, t)=\min\{d_{\min}(x, y): x, y\in X\}.$$
\begin{enumerate}
\item If $d_{\min}(s, x)=d_{\min}(t, x)$ for all $x, y\in X$ (so $\{s, t\}$ is a cherry), then
\begin{enumerate}[(i)]
\item Reduce $t$ in $\cD_{\min}$ to give the distance matrix $\cD'$ on $X'=X-\{t\}$.

\item Apply {\sc Equidistant Normal} to input $X'$ and $\cD'$. Construct $(\cN_0, w_0)$ from the returned weighted phylogenetic network $(\cN'_0, w'_0)$ on $X'$ as follows. If $p'_s$ is the parent of $s$ in $(\cN'_0, w'_0)$, then subdivide $(p'_s, s)$ with a new vertex $p_s$, adjoin a new leaf $t$ to $p_s$ via the new edge $(p_s, t)$, and set
$$w_0(p_s, s)=w_0(p_s, t)={\textstyle \frac{1}{2}}d_{\min}(s, t)$$
and $w_0(p'_s, p_s)=w'_0(p'_s, s)-\frac{1}{2}d_{\min}(s, t)$. Keeping all other edge weights the same as their counterparts in $(\cN'_0, w'_0)$, return $(\cN_0, w_0)$.
\end{enumerate}

\item Else ($\{s, t\}$ is a reticulated cherry, in which case, $t$ is the reticulation leaf if there exists an $x\in X-\{s, t\}$ such that $d_{\min}(t, x) < d_{\min}(s, x)$),
\begin{enumerate}[(i)]
\item Cut $\{s, t\}$ in $\cD_{\min}$ to give the distance matrix $\cD'$ on $X$.

\item Apply {\sc Equidistant Normal} to input $X$ and $\cD'$. Construct $(\cN_0, w_0)$ from the returned weighted phylogenetic network $(\cN'_0, w'_0)$ on $X$ as follows. If $p'_s$ and $p'_t$ denote the parents of $s$ and $t$ in $(\cN'_0, w'_0)$, respectively, then subdivide $(p'_s, s)$ and $(p'_t, t)$ with new vertices $p_s$ and $p_t$, respectively, adjoin $p_s$ and $p_t$ via the new edge $(p_s, p_t)$, set $w_0(p_s, s)=\frac{1}{2}d_{\min}(s, t)$ and $w_0(p'_s, p_s)=w'_0(p'_s, s)-\frac{1}{2}d_{\min}(s, t)$, and, for some positive real value $\omega$ such that $\omega\le \frac{1}{2}d_{\min}(s, t)$ and $\omega\le w'_0(p'_t, t)$, set $w_0(p_t, t)=\omega$, $w_0(p_s, p_t)=\frac{1}{2}d_{\min}(s, t)-\omega$, and $w_0(p'_t, p_t)=w'_0(p'_t, t)-\omega$. Keeping all other edge weights the same as their counterparts in $(\cN'_0, w'_0)$, return $(\cN_0, w_0)$.
\end{enumerate}
\end{enumerate}
\end{enumerate}

We now consider the running time of {\sc Equidistant Normal}. The algorithm takes as input a set $X$ and an $|X|\times |X|$ distance matrix $\cD_{\min}$ whose entries are the minimum length of an up-down path joining elements in $X$ of an equidistant-weighted normal network $(\cN, w)$ on $X$. Unless $|X|\in \{1, 2\}$, in which case {\sc Equidistant Normal} runs in constant time, each iteration starts by finding a $2$-element subset $\{s, t\}$ of $X$ such that
$$d_{\min}(s, t)=\min\{d_{\min}(x, y): x, y\in X\}.$$
This takes $O(|X|^2)$ time. Once such a $2$-element subset is found, we compute $\cD'$. This computation is done in one of two ways depending on whether or not
$$d_{\min}(s, x)=d_{\min}(t, x)$$
for all $x\in X-\{s, t\}$. If, for some $x$,
$$d_{\min}(s, x)\neq d_{\min}(t, x),$$
we need to additionally check which of $d_{\min}(s, x) < d_{\min}(t, x)$ and $d_{\min}(s, x) > d_{\min}(t, x)$ hold. Thus the determination of which way to compute $\cD'$ can be done in $O(|X|)$ time. Regardless of the way, $\cD'$ can be computed in $O(|X|)$ time. Once $(\cN'_0, w'_0)$ is returned, it can be augmented to $(\cN_0, w_0)$ in constant time. Hence the total time of each iteration is $O(|X|^2)$ time.

When we recurse, the distance matrix $\cD'$ inputted to the recursive call is the minimum distance matrix of a normal network with either one less leaf or one less reticulation than a normal network for which $\cD_{\min}$ is the minimum distance matrix. Since a normal network has at most $|X|-2$ reticulations, it has $O(|X|)$ vertices in total~\cite{bic12} (also see~\cite{mcd15}), and so the total number of iterations is at most $O(|X|)$. Thus {\sc Equidistant Normal} completes in $O(|X|^3)$ time. This completes the proof of Theorem~\ref{child1}.

\section{Proof of Theorem~\ref{child2}}
\label{proof2}

In this section, we prove Theorem~\ref{child2}. We begin with two lemmas. Let $\cN$ be a phylogenetic network, and suppose that $\{s, t\}$ is either a cherry or a reticulated cherry in which $t$ is the reticulation leaf in $\cN$. We refer to the parent of $s$ as the {\em tree vertex of $\{s, t\}$}. For the next lemma, the proof of (i) and (iii) are given in~\cite{bor17}, while the proof of (ii) is similar to that of Lemma~\ref{cherry1} and is omitted.

\begin{lemma}
Let $|X|\ge 2$, and let $(\cN, w)$ be a weighted tree-child network on $X\cup \{r\}$, where $r$ is an outgroup. Let $\cD_{\min}$ be the minimum distance matrix and let $\bd_{\max}$ be the maximum distance outgroup vector of $(\cN, w)$. Let $\{s, t\}$ be a $2$-element subset of $X$ such that
$$\cQ_r(s, t)=\max\{\cQ_r(x, y): x, y\in X\},$$
where $\cQ_r(x, y)=\frac{1}{2}\{d_{\max}(r, x)+d_{\max}(r, y)-d_{\min}(x, y)\}$. Then
\begin{enumerate}[{\rm (i)}]
\item $\{s, t\}$ is either a cherry or a reticulated cherry of $(\cN, w)$.

\item $\{s, t\}$ is a cherry of $(\cN, w)$ if and only if
$$d_{\min}(s, x)+d_{\max}(r, t)-d_{\max}(r, s)=d_{\min}(t, x)$$
for all $x\in X-\{s, t\}$. Otherwise, $\{s, t\}$ is a reticulated cherry in which $t$ is the reticulation leaf if
$$d_{\min}(s, x)+d_{\max}(r, t)-d_{\max}(r, s) > d_{\min}(t, x)$$
for some $x\in X-\{s, t\}$.

\item The length of the longest up-down path in $(\cN, w)$ starting at $r$ and ending at the tree vertex of $\{s, t\}$ is $\cQ_r(s, t)$, and $d_{\max}(r, s)$ and $d_{\max}(r, t)$ are each realised by up-down paths that include this tree vertex.
\end{enumerate}
\label{cherry2}
\end{lemma}

\begin{lemma}
Let $|X|\ge 2$, and let $(\cN, w)$ be a reticulation-pair weighted normal network on $X\cup \{r\}$, where $r$ is an outgroup. Let $\cD_{\min}$ and $\bd_{\max}$ be the minimum distance matrix and maximum distance outgroup vector of $(\cN, w)$, respectively. Let $\{s, t\}$ be a $2$-element subset of $X$ such that
$$\cQ_r(s, t)=\max\{\cQ_r(x, y): x, y\in X\}.$$
Then the distance matrix and distance vector obtained from $\cD_{\min}$ and $\bd_{\max}$ by reducing $t$ if $\{s, t\}$ is a cherry and by isolating $\{s, t\}$ if $\{s, t\}$ is a reticulated cherry in which $t$ is the reticulation leaf are the minimum distance matrix $\cD'_{\min}$ and maximum distance outgroup vector $\bd'_{\max}$ of the weighted network obtained from $(\cN, w)$ by reducing $t$ if $\{s, t\}$ is a cherry and isolating $\{s, t\}$ if $\{s, t\}$ is a reticulated cherry in which $t$ is the reticulation leaf.
\label{isolating}
\end{lemma}

\begin{proof}
By Lemma~\ref{cherry2}, $\{s, t\}$ is either a cherry or a reticulated cherry of $(\cN, w)$. By the same lemma, if $\{s, t\}$ is a reticulated cherry, then we may assume, without loss of generality, that $t$ is the reticulation leaf. Furthermore, it follows by Lemma~\ref{unchanged} that $(\cN', w')$ is a reticulation-pair normal network with outgroup $r$. If $\{s, t\}$ is a cherry, then it is clear that the lemma holds. Therefore, suppose that $\{s, t\}$ is a reticulated cherry, in which case $(\cN', w')$ is obtained from $(\cN, w)$ by isolating $\{s, t\}$. Let $\cD'$ and $\bd'$ be the distance matrix and distance vector obtained from $\cD_{\min}$ and $\bd_{\max}$ by isolating $\{s, t\}$. We will show that $\cD'$ and $\bd'$ are the minimum distance matrix $\cD'_{\min}$ and maximum distance outgroup vector $\bd'_{\max}$ of $(\cN', w')$.

Let $p_s$ and $p_t$ denote the parents of $s$ and $t$, respectively, and let $g_t$ denote the parent of $p_t$ that is not $p_s$ in $(\cN, w)$. Note that, as $\cN$ is normal, $g_t$ is a tree vertex and not an ancestor of $p_s$. Since the only up-down paths in $(\cN, w)$ joining elements in $X$ which traverse $(g_t, p_t)$ involve $t$, it follows that to complete the proof, it suffices to show that $d'(t, x)=d'_{\min}(t, x)$ for all $x\in X-\{t\}$ and $d'(r, t)=d'_{\max}(r, t)$.

By Lemma~\ref{cherry2}(iii),
$$d'(r, t)=d'_{\max}(r, t).$$
Furthermore, let $\gamma=d_{\max}(r, t)-d_{\max}(r, s)$. Then, by Lemma~\ref{cherry2},
$$d'_{\min}(t, x)=d'_{\min}(s, x)+\gamma$$
for all $x\in X-\{s, t\}$, so $d'(t, x)=d'_{\min}(t, x)$ for all $x\in X-\{s, t\}$. Lastly, as $w$ is a reticulation-pair weighting, $w(g_t, p_t)=w(p_s, p_t)$, so $d'(t, s)=d'_{\min}(t, s)$. This completes the proof of the lemma.
\end{proof}

We next establish the uniqueness part of Theorem~\ref{child2}.

\begin{proof}[Proof of the uniqueness part of Theorem~\ref{child2}]
The proof is by induction on the sum of the number $n$ of leaves and the number $k$ of reticulations in $(\cN, w)$. If $n+k=1$, then $n=1$ and $k=0$, and so $(\cN, w)$ consists of the single vertex in $X$, in which case the uniqueness holds. If $n+k=2$, then $n=2$, $k=0$, and $(\cN, w)$ consists of two leaves attached to the root, one of which is the outgroup $r$. Again, the uniqueness holds. Now suppose that $n+k\ge 3$, so $n\ge 3$, and the uniqueness holds for all reticulation-pair weighted normal networks for which the sum of the number of leaves and the number of reticulations is at most $n+k-1$.

Let $\{s, t\}$ be a $2$-element subset of $X$ such that
$$\cQ_r(s, t)=\max\{\cQ_r(x, y): x, y\in X\}.$$
By Lemma~\ref{cherry2}, $\{s, t\}$ is either a cherry or a reticulated cherry. If $\{s, t\}$ is a reticulated cherry, then, by the same lemma, $\cD_{\min}$ and $\bd_{\max}$ determine whether $s$ or $t$ is the reticulation leaf. Thus, without loss of generality, we may assume that $t$ is the reticulation leaf. Let $(\cN', w')$, $\cD'$, and $\bd'$ be the weighted phylogenetic network, distance matrix, and distance vector obtained from $(\cN, w)$, $\cD_{\min}$, and $\bd_{\max}$, respectively, by reducing $t$ if $\{s, t\}$ is a cherry or isolating $\{s, t\}$ if $\{s, t\}$ is a reticulated cherry. Now $(\cN', w')$ either has $n-1$ leaves and $k$ reticulations, or $n$ leaves and $k-1$ reticulations, and so, by Lemma~\ref{isolating} and the induction assumption, up to equivalence, $(\cN', w')$ is the unique reticulation-pair weighted normal network with outgroup $r$ whose minimum distance matrix is $\cD'$ and maximum distance outgroup vector is $\bd'$.

Let $(\cN_1, w_1)$ be a reticulation-pair weighted normal network on $X$ with outgroup $r$ whose minimum distance matrix is $\cD_{\min}$ and maximum distance outgroup vector is $\bd_{\max}$. By Lemma~\ref{cherry2}, $\{s, t\}$ is either a cherry or a reticulated cherry in $(\cN_1, w_1)$. Moreover, by the same lemma, $\{s, t\}$ is a cherry in $(\cN_1, w_1)$ if and only if it is a cherry in $(\cN, w)$. First assume that $\{s, t\}$ is a cherry in $(\cN, w)$. Let $(\cN'_1, w'_1)$ be the reticulation-pair weighted normal network with outgroup $r$ obtained from $(\cN_1, w_1)$ by reducing $t$. By Lemma~\ref{isolating}, $\cD'$ and $\bd'$ are the minimum distance matrix and maximum distance outgroup vector of $(\cN'_1, w'_1)$. Thus, by the induction assumption, $(\cN'_1, w'_1)$ and $(\cN', w')$ are equivalent. Using this equivalence and considering $d_{\min}(s, t)$, it is easily checked that $(\cN_1, w_1)$ and $(\cN, w)$ are equivalent.

Now assume that $\{s, t\}$ is a reticulated cherry in $(\cN, w)$. Then $\{s, t\}$ is a reticulated cherry in $(\cN_1, w_1)$ where, by Lemma~\ref{cherry2}, $t$ is the reticulation leaf. Let $(\cN'_1, w'_1)$ be the reticulation-pair weighted normal network with outgroup $r$ obtained from $(\cN_1, w_1)$ by isolating $\{s, t\}$. Since $\cD_{\min}$ and $\bd_{\max}$ are the minimum distance matrix and maximum distance outgroup vector of $(\cN_1, w_1)$, it follows by Lemma~\ref{isolating} that $\cD'$ and $\bd'$ are the minimum distance matrix and maximum distance outgroup vector of $(\cN'_1, w'_1)$. So, by the induction assumption, $(\cN'_1, w'_1)$ and $(\cN', w')$ are equivalent.

In $(\cN, w)$, let $p_s$ and $p_t$ denote the parents of $s$ and $t$, respectively, and let $g_t$ denote the parent of $p_t$ that is not $p_s$. Since $\cN$ is normal, $g_t$ is a tree vertex and not an ancestor of $p_s$. We next show that there is precisely one choice for the attachment of the edge in $(\cN', w')$, and thus also in $(\cN'_1, w'_1)$, corresponding to $(g_t, p_t)$ in $(\cN, w)$.

Since $\cN$ is normal, there is a (directed) path $P$ from $g_t$ to a leaf,  say $\ell$, containing no reticulations. Since $w$ is reticulation-pair, $d_{\min}(t, \ell)$ is the length of the up-down path whose union of edges consists of the edges in $\{(g_t, p_t), (p_t, t)\}$ and $P$. Thus, if we knew $\ell$, then, to locate the place in $(\cN', w')$ at which to insert $g_t$, we simply start at $\ell$ and follow the unique path against the direction of the edges towards the root until we reach a distance
$$d_{\min}(t, \ell)-(d_{\max}(r, t)-\cQ_r(s, t))$$
from $\ell$, since the bracketed term gives the combined length of $(g_t, p_t)$ and $(p_t, t)$. However, {\em a priori}, we do not know $\ell$. So there are potentially $O(n)$ places in $(\cN', w')$ at which we could insert $g_t$. We claim there is exactly one such place to insert $g_t$ so that the resulting weighted network (after subdividing the edge incident with $t$, inserting a new vertex $p_t$ and adding the new edge $(g_t, p_t)$) has minimum distance matrix $\cD_{\min}$ and maximum distance outgroup vector $\bd_{\max}$ and no zero-length tree-edges.

We call a leaf $\ell'$ a \emph{candidate leaf} if
\begin{itemize}
\item the path starting at $\ell'$ and going against the direction of the edges (and, thus, towards the root) a distance
$d_{\min}(t, \ell')-(d_{\max}(r, t)-\cQ_r(s, t))$ does not traverse a reticulation;
\item the unique position along this path at a distance $d_{\min}(t, \ell')-(d_{\max}(r, t)-\cQ_r(s, t))$ from $\ell'$, denoted $g_{\ell'}$, is not a vertex, that is, $g_{\ell'}$ is partway along an edge of $(\cN', w')$; and
\item $g_{\ell'}$ is not an ancestor of $p_s$.
\end{itemize}
Note that the unknown leaf $\ell$ is a candidate leaf. Moreover, if the second or third conditions were not satisfied and we tried to reconstruct a network by inserting $g_t$ at position $g_{\ell'}$ we would either need to introduce zero-weight tree edges or we would introduce a shortcut, contradicting the assumptions about $(\cN, w)$.

We now show that if $g_{\ell'}$ is not at the same position as $g_t$, then $g_{\ell'}$ is an ancestor of $g_t$. Suppose not, then a minimum length up-down path in $(\cN, w)$ from $t$ to $\ell'$ via $g_{t}$ must traverse the edge containing position $g_{\ell'}$. But then the length of this up-down path is not $d_{\min}(t, \ell')$, by definition of $g_{\ell'}$. Likewise, a minimum length up-down path in $(\cN, w)$ from $t$ to $\ell'$ via $p_{s}$ must traverse the edge containing position $g_{\ell'}$ (since $p_s$ is not a descendant of $g_{\ell'}$). Again the length of this up-down path is not $d_{\min}(t, \ell')$. This contradicts the fact that $(\cN, w)$ has minimum distance matrix $\cD_{\min}$.

Finally, observe that if $g_{\ell'}$ is an ancestor of $g_t$, then the minimum path length between $t$ and $\ell$ in the network obtained from $(\cN', w')$ by adding an edge $(g_{\ell'}, p_t)$ will be strictly larger than $d_{\min}(t, \ell)$. This establishes the claim. Moreover, the correct position $g_t$ can be found as the unique common descendant of all candidate positions $g_{\ell'}$. It now follows that, as $(\cN'_1, w'_1)$ and $(\cN', w')$ are equivalent, $(\cN_1, w_1)$ and $(\cN, w)$ are equivalent. This completes the proof of the uniqueness part of Theorem~\ref{child2}.

\end{proof}

\subsection{The Algorithm}

Let $(\cN, w)$ be a reticulation-pair weighted normal network on $X\cup \{r\}$, where $r$ is an outgroup, and let $\cD_{\min}$ and $\bd_{\max}$ denote the minimum distance matrix and maximum distance outgroup vector of $(\cN, w)$. The following algorithm, called {\sc Reticulation-Pair Normal}, takes as input $X$, $\cD_{\min}$, and $\bd_{\max}$ and returns a weighted network $(\cN_0, w_0)$ equivalent to $(\cN, w)$ in which all reticulation edges have weight zero. As before, the proof of its correctness is essentially established in proving the uniqueness part of the theorem and so is omitted. But its running time is given at the end.

\begin{enumerate}[1.]
\item If $|X|=1$, then return the weighted phylogenetic network consisting of the single vertex in $X$.

\item If $|X|=2$, say $X=\{r, s\}$, then return the phylogenetic network $(\cN_0, w_0)$ consisting of leaves $r$ and $s$ adjoined to the root $\rho$ with $(\rho, r)$ and $(\rho, s)$ positively weighted so that $d_{\max}(r, s)=w(\rho, r)+w(\rho, s)$.

\item Else, find a $2$-element subset $\{s, t\}$ of $X$ such that
$$\cQ_r(s, t)=\max\{\cQ_r(x, y): x, y\in X\}.$$
\begin{enumerate}[(a)]
\item If $d_{\min}(s, x)+d_{\max}(r, t)-d_{\max}(r, s)=d_{\min}(t, x)$ for all $x\in X$ (so $\{s, t\}$ is a cherry), then
\begin{enumerate}[(i)]
\item Reduce $t$ in $\cD_{\min}$ and $\bd_{\max}$ to give the distance matrix $\cD'$ and distance vector $\bd'$ on $X'=X-\{t\}$.

\item Apply {\sc Reticulation-Pair Normal} to input $X'\cup \{r\}$, $\cD'$, and $\bd'$. Construct $(\cN_0, w_0)$ from the returned weighted phylogenetic network $(\cN'_0, w'_0)$ on $X'$ as follows. If $p'_s$ is the parent of $s$ in $(\cN'_0, w'_0)$, then subdivide $(p'_s, s)$ with a new vertex $p_s$, adjoin a new leaf $t$ to $p_s$ via a new edge $(p_s, t)$, and set
$$w_0(p_s, s)=d_{\max}(r, s)-\cQ_r(s, t),$$
$$w_0(p'_s, p_s)=w'_0(p'_s, s)-w_0(p_s, s),$$
and
$$w_0(p_s, t)=d_{\min}(s, t)-w_0(p_s, s).$$
Keeping all other edges weight the same as their counterparts in $(\cN'_0, w'_0)$, return $(\cN_0, w_0)$.
\end{enumerate}

\item Else ($\{s, t\}$ is a reticulated cherry, in which $t$ is the reticulation leaf if there exists an $x\in X-\{s, t\}$ such that $d_{\min}(t, x) < d_{\min}(s, x)+d_{\max}(r, t)-d_{\max}(r, s)$),
\begin{enumerate}[(i)]
\item Isolate $\{s, t\}$ in $\cD_{\min}$ and $\bd_{\max}$ to give the distance matrix $\cD'$ and distance vector $\bd'$ on $X$.

\item Apply {\sc Reticulation-Pair Normal} to input $X\cup \{r\}$, $\cD'$, and $\bd'$. Construct $(\cN_0, w_0)$ from the returned weighted phylogenetic network $(\cN'_0, w'_0)$ on $X\cup \{r\}$ as follows. For each leaf $\ell$ in $X-\{s, t\}$, follow the unique path starting at $\ell$ and going against the direction of the edges towards the root until either a reticulation or a distance 
$$d_{\min}(t, \ell)-(d_{\max}(r, t)-\cQ_r(s, t))$$
from $\ell$ is reached. Amongst the points reached, insert a new vertex $g_t$ in the unique point that is a descendant of all the other points reached and  weight the edges incident with $g_t$ appropriately. Now, subdivide the edge incident with $t$ with a new vertex $p_t$, add the new edge $(g_t, p_t)$, and set $w_0(p_s, p_t)=0$, $w_0(g_t, p_t)=0$, and
$$w_0(p_t, t)=d_{\max}(r, t)-\cQ_r(s, t),$$
where $p_s$ is the parent of $s$. Keeping all other edges the same weight as their counterparts in $(\cN'_0, w'_0)$, return $(\cN_0, w_0)$.
\end{enumerate}
\end{enumerate}
\end{enumerate}

For the running time, {\sc Reticulation-Pair Normal} takes as input a set $X\cup \{r\}$, a $|X|\times |X|$ distance matrix $\cD_{\min}$ whose entries are the minimum-length of an up-down path joining elements in $X$, and a distance vector $\bd_{\max}$ of length $|X|$ whose entries are the maximum length of an up-down path from $r$ to each element in $X$ of a reticulation-pair normal network $(\cN, w)$ on $X\cup \{r\}$, where $r$ is an outgroup. If $|X|\in \{1, 2\}$, then the algorithm runs in constant time. If $|X|\ge 3$, each iteration begins by finding a $2$-element subset $\{s, t\}$ of $X$ such that
$$\cQ_r(s, t)=\max\{\cQ_r(x, y): x, y\in X\}.$$
This takes $O(|X|^2)$ time, and once such a $2$-element subset is found, we construct $\cD'$ and $\bd'$. This construction is done in one of two ways depending on whether or not
$$d_{\min}(s, x)+d_{\max}(r, t)-d_{\max}(r, s)=d_{\min}(t, x)$$
for all $x\in X-\{s, t\}$. If, for some $x$,
$$d_{\min}(s, x)+d_{\max}(r, t)-d_{\max}(r, s)\neq d_{\min}(t, x),$$
we need to additionally check which of
$$d_{\min}(s, x)+d_{\max}(r, t)-d_{\max}(r, s) < d_{\min}(t, x)$$
and
$$d_{\min}(s, x)+d_{\max}(r, t)-d_{\max}(r, s) > d_{\min}(t, x)$$
holds. Thus the determination of which way to construct $\cD'$ and $\bd'$ can be done in $O(|X|)$ time. Whether $\cD'$ and $\bd'$ is constructed by reducing an element of $X$ or isolating a $2$-element subset of $X$, the construction can be done in $O(|X|)$ time. Once $(\cN'_0, w'_0)$ is returned, it takes constant time to augment to $(\cN_0, w_0)$ if $\cD'$ and $\bd'$ have been obtained from $\cD_{\min}$ and $\bd_{\max}$ by reducing. Otherwise, we need to find the unique place in $(\cN'_0, w'_0)$ to insert $g_t$. Since $(\cN, w)$ has $O(|X|)$ vertices in total~\cite{bic12}, it takes at most $O(|X|^2)$ time to find the possible locations in which to insert $g_t$. Finding the correct location, the one that is a descendant of all the others, can be done in $O(|X|)$ time by repeatedly deleting vertices of out-degree zero until the first possible location appears as a vertex of out-degree zero. Thus $(\cN_0, w_0)$ can be returned in $O(|X|^2)$ time, and so the total time of each iteration is $O(|X|^2)$.

When we recurse, the distance matrix $\cD'$ and distance vector $\bd'$ inputted to the recursive call is the minimum distance matrix and maximum distance outgroup vector of a normal network with either one less leaf or one less reticulation than $(\cN, w)$. As normal networks have at most $|X|-2$ reticulations, and therefore $O(|X|)$ vertices in total~\cite{bic12}, the total number of iterations is at most $O(|X|)$. Hence {\sc Reticulation-Pair Normal} completes in $O(|X|^3)$ time, thereby completing the proof of Theorem~\ref{child2}.

\section*{Acknowledgements}

Katharina Huber and Vincent Moulton thank the Biomathematics Research Centre at the University of Canterbury for its hospitality.

\end{document}